\documentclass[twoside,leqno]{article}

\usepackage[letterpaper]{geometry}

\usepackage{ltexpprt}
\usepackage{hyperref}
\usepackage{amsmath}
\usepackage{amssymb}
\usepackage{mathtools}
\usepackage{bm}
\usepackage{xcolor}
\usepackage{enumerate}
\usepackage{enumitem}
\usepackage{algorithm}  
\usepackage[version=4]{mhchem}
\newcommand{\R}{\mathbb{R}}
\newcommand{\I}{\mathbb{I}}

\newcommand{\bff}[1]{\mathbf{#1}}

\renewcommand{\L}{\mathrm{L}}
\newcommand{\U}{\mathrm{U}}
\newcommand{\cv}{\mathrm{cv}}
\newcommand{\cc}{\mathrm{cc}}
\newcommand{\A}{\mathrm{A}}
\newcommand{\B}{\mathrm{B}}

\newcommand{\T}{\mathrm{T}}

\newcommand{\mrm}[1]{\mathrm{#1}}
\newcommand{\ysnote}[1]{{#1}}

\newtheorem{Assumption}{Assumption}[section]
\newtheorem{EExample}{Example}[section]
\begin{document}

\newcommand\relatedversion{}

\title{\Large  Computing AD-compatible subgradients of convex relaxations of implicit functions\relatedversion}
\author{Yingkai Song\thanks{Andlinger Center for Energy and the Environment, Princeton University.}
	\and Kamil A. Khan\thanks{Department of Chemical Engineering, McMaster University. \newline This work was supported in part by the Natural Sciences and Engineering Research Council of Canada (NSERC)
under Grant RGPIN-2017-05944.}}

\date{}

\maketitle

% Copyright Statement
% When submitting your final paper to a SIAM proceedings, it is requested that you include
% the appropriate copyright in the footer of the paper.  The copyright added should be
% consistent with the copyright selected on the copyright form submitted with the paper.
% Please note that "20XX" should be changed to the year of the meeting.

% Default Copyright Statement
\fancyfoot[R]{\scriptsize{Copyright \textcopyright\ 2025 by SIAM\\
Unauthorized reproduction of this article is prohibited}}

% Depending on which copyright you agree to when you sign the copyright form, the copyright
% can be changed to one of the following after commenting out the default copyright statement
% above.

%\fancyfoot[R]{\scriptsize{Copyright \textcopyright\ 20XX\\
%Copyright for this paper is retained by authors}}

%\fancyfoot[R]{\scriptsize{Copyright \textcopyright\ 20XX\\
%Copyright retained by principal author's organization}}

%\pagenumbering{arabic}
%\setcounter{page}{1}%Leave this line commented out.

\begin{abstract} \small\baselineskip=9pt 
Automatic generation of convex relaxations and subgradients is critical in global optimization, and is typically carried out using variants of automatic/algorithmic differentiation (AD).
At previous AD conferences, variants of the forward and reverse AD modes were presented to evaluate accurate subgradients for convex relaxations of supplied composite functions. In a recent approach for generating convex relaxations of implicit functions, these relaxations are constructed as optimal-value functions; this formulation is versatile but complicates sensitivity analysis. We present the first subgradient propagation rules for these implicit function relaxations, based on supplied AD-like knowledge of the residual function. \ysnote{Our new subgradient rules allow implicit function relaxations to be added to the elemental function libraries for the forward AD modes for subgradient propagation of convex relaxations.} Proof-of-concept \ysnote{numerical} results in Julia are presented.
\end{abstract}
\section{Introduction}

Typical state-of-the-art methods for global minimization proceed by generating lower and upper bounds on the unknown optimal value, and then successively improving these bounds. While upper bounds may be evaluated as local minima using standard solvers, typical approaches for computing useful lower bounds (e.g.~in the solver BARON~\cite{sahinidis1996baron}) involve generating convex relaxations and minimizing those. Based on AD-like construction rules by McCormick~\cite{mccormick1976computability}, a forward AD mode variant by Mitsos et al.~\cite{mitsos2009mccormick} automatically evaluates a convex relaxation for a supplied composition of elemental functions from a library. Mitsos et al.~also developed a forward AD mode variant to evaluate subgradients for these relaxations; subgradients are useful in convex optimization, and are particularly useful when computing the lower bounds in global optimization. At the AD2012 conference, Beckers et al.~\cite{beckers2012adjoint} presented an efficient reverse AD mode variant of Mitsos et al.'s subgradient propagation method. These AD modes are in active use today, including in the Julia package \textsf{EAGO.jl}~\cite{wilhelm2022eago} for global optimization, and in a new approach for computing adjoint subgradients of convex relaxations of differential equation solutions~\cite{zhang2024adjoint}.

Cao and Khan~\cite{cao2023general} recently presented a new method to construct useful convex relaxations of implicit functions based on known relaxations of the residual functions. Though this approach is versatile, it constructs implicit function relaxations as convex optimal-value functions, for which sensitivity analysis is traditionally considered to be difficult~\cite{fiacco1990sensitivity}. In order to incorporate implicit functions into the forward AD modes for subgradient propagation, we seek subgradients of these new relaxations. \ysnote{This article presents the first useful subgradient computation rules for these relaxations, thus allowing implicit functions to be added to the elemental function libraries for constructing McCormick-style convex relaxations with subgradients computed by forward-mode AD, and to be used in established implementations. We informally refer to these properties as ``AD-compatibility''. While AD-compatibility is generally not an issue for compositions of smooth functions due to the structure of the well-known chain-rule, appropriate AD rules only have been developed recently for composite nonsmooth functions~\cite{beckers2012adjoint,khan2018branch,KhanBartonFwdAD,mitsos2009mccormick}. We nevertheless note that due to the nonlinearity with respect to input directions of one of our new approaches, there may not be a direct reverse AD analog of that approach; a similar situation was discussed in~\cite{griewank1995automatic}.} We proceed based on recent AD rules for nonsmooth optimal value functions~\cite{song2024generalized}, bound-constrained convex functions~\cite{yuan2023automatic}, and nonsmooth closed-form composite functions~\cite{KhanBartonAD,KhanBartonFwdAD}, the latter of which were presented at the AD2012 and AD2016 conferences.

The remainder of this article is laid out as follows. Section~\ref{sec:preliminaries} summarizes relevant established concepts including convex relaxations, subgradients, and the new implicit function relaxations of~\cite{cao2023general}. Section~\ref{sec:outline} summarizes the particular situations we will later consider. Section~\ref{sec:nx=1} presents new subgradients for relaxations of scalar-valued implicit functions, Section~\ref{sec:nx>1} covers vector-valued implicit functions, and Section~\ref{sec:numerical} presents numerical examples based on a proof-of-concept implementation in Julia.
\section{Preliminaries}\label{sec:preliminaries}
%\subsection{Convex relaxations for implicit functions}
\subsection{Notation}
Until Section~\ref{sec:numerical}, scalars are denoted as lowercase letters (e.g. $x\in\R$), vectors are denoted as boldface lowercase letters (e.g. $\bff{x}\in\R^n$), and the $i^{\mrm{th}}$ component of a vector $\bff{x}$ is denoted as $x_i$. $\bff{0}$ denotes a vector with all components zero, with a dimension that is clear from the context.  For each $j\in \{1,2,...,n\}$, denote the $j$\textsuperscript{th} unit coordinate vector in $\R^n$ as $\bff{e}^{(j)}$. Inequalities involving vectors are to be interpreted componentwise. Matrices are denoted as boldface uppercase letters (e.g. $\bff{M}\in\R^{n\times m}$). The $j^{\mrm{th}}$ column of a matrix $\bff{M}$ is denoted as $\bff{m}_{(j)}$. Sets are denoted as uppercase letters (e.g. $X\subset\R^n$). For any $\bff{x}^\L,\bff{x}^\U\in\R^n$ with $\bff{x}^\L\le \bff{x}^\U$, an interval $X:=[\bff{x}^\L,\bff{x}^\U]$ is defined as the compact set $\{\bm{\xi}\in\R^n:\bff{x}^\L\le \bm{\xi}\le \bff{x}^\U\}$. Let $\I\R^{n}$ denote the collection of all intervals in $\R^n$. The interior of a set $X$ is denoted as $\mathrm{int}(X)$. For a scalar-valued \emph{(Fr\'echet) differentiable} function $f(\bff{x},\bff{y})$, $\nabla_\bff{x}f(\bff{x},\bff{y})\in\R^{n_x}$ and $\nabla_\bff{y}f(\bff{x},\bff{y})\in\R^{n_y}$ denote the partial gradients of $f$ with respect to $\bff{x}$ and $\bff{y}$, respectively. A vector function $\bff{f}$ is said to be convex, if each component $f_i$ is convex. The abbreviation ``s.t.'' stands for ``subject~to''.

\subsection{Convex relaxations and subgradients}

In~state-of-the-art  global optimization solvers such as BARON~\cite{sahinidis1996baron}, upper and lower bounds are generated for the unknown optimal value, and are then successively refined. In this setting, lower bounds are typically generated by minimizing convex relaxations; the following definition of convex relaxations accommodates functions with irregular domains.

\begin{Definition}[from \cite{cao2023general}]\label{Definition:convex}
	\normalfont
	Consider a convex set $P\subset \R^{n_p}$ and a subset $Q\subset P$. Consider the extended real numbers $\bar{\R}:=\R\cup\{-\infty,+\infty\}$, and a function $\bm{\phi}:Q\to{\R}^m$. Then:
	\begin{enumerate}
		\item $\bm{\phi}^\cv:P\to\bar{\R}^m$ is a \emph{convex relaxation} of $\bm{\phi}$ on $P$ if
		\begin{itemize}
			\item $\bm{\phi}^\cv(\bff{p})\le \bm{\phi}(\bff{p})$ for all $\bff{p}\in Q$, and 
			\item $\bm{\phi}^\cv$ is convex on $P$.
		\end{itemize}
		\item $\bm{\phi}^\cc:P\to\bar{\R}^m$ is a \emph{concave relaxation} of $\bm{\phi}$ on $P$ if
		\begin{itemize}
			\item $\bm{\phi}^\cc(\bff{p})\ge \bm{\phi}(\bff{p})$ for all $\bff{p}\in Q$, and 
			\item $\bm{\phi}^\cc$ is concave on $P$.
		\end{itemize}
	\end{enumerate}
\end{Definition}

\begin{Definition}[from {\cite{scholtes2012introduction}}]\label{Definition:direcderiv}
	\normalfont
	Given an open set $X\subset \R^n$, a function $\bff{f}:X\to\R^m$, {and vectors $\bff{x}\in X$ and $\bff{d}\in\R^n$}, the limit
	\begin{equation*}
		\lim_{\alpha\downarrow 0}\frac{\bff{f}(\bff{x}+\alpha\bff{d})-\bff{f}(\bff{x})}{\alpha}
	\end{equation*}
	is called the \emph{directional derivative} of $\bff{f}$ at $\bff{x}$ {in the direction $\bff{d}$} if it exists, and is denoted as {$\bff{f}'(\bff{x};\bff{d})$ or $[\bff{f}]'(\bff{x};\bff{d})$; the latter notation is clearer when a function's name includes superscripts.} The function $\bff{f}$ is \emph{directionally differentiable} at $\bff{x}$ if $\bff{f}'(\bff{x};\bff{d})$ exists and is finite for each $\bff{d}\in \R^n$.
\end{Definition}
{Convex functions defined on open sets are always directionally differentiable (see \cite[Theorem~23.4]{rockafellar2015convex}).} The following definition of subgradients is standard in convex analysis~\cite{rockafellar2015convex}.
\begin{Definition}[{from {\cite[Page 214]{rockafellar2015convex}}}]\label{Definition:subgradient}
	\normalfont
	Given an open convex set $X\subset \R^n$, a convex function $f^{\cv}:X\to\R$, \ysnote{and a point $\bff{x}\in\mathrm{int}(X)$}, a vector $\bff{s}^{\cv}\in\R^{n}$ is a \emph{subgradient} of $f^{\cv}$ at $\bff{x}\in X$ if
	\begin{equation*}
		f^{\cv}(\bm{\xi})\ge f^{\cv}(\bff{x})+\langle\bff{s}^{\cv},\bm{\xi}-\bff{x}\rangle,\quad\forall \bm{\xi}\in X.
	\end{equation*}
	Similarly, given a concave function $f^{\cc}:X\to\R$, $\bff{s}^{\cc}\in\R^n$ is a subgradient of $f^{\cc}$ at $\bff{x}\in X$ if
	\begin{equation*}
		f^{\cc}(\bm{\xi})\le f^{\cc}(\bff{x})+\langle\bff{s}^{\cc},\bm{\xi}-\bff{x}\rangle,\quad\forall \bm{\xi}\in X.
	\end{equation*}
	The \emph{subdifferential} $\partial f^{\cv}(\bff{x})\subset\R^n$ (resp. $\partial f^{\cc}(\bff{x})\subset\R^n$) is the collection of all subgradients of $f^{\cv}$ (resp. $f^{\cc}$) at $\bff{x}$, and is nonempty, convex, and compact.
\end{Definition}

\ysnote{As summarized in the introduction, the McCormick relaxations~\cite{mccormick1976computability,mitsos2009mccormick} are convex relaxations that may be computed by a forward sweep through a composite function's computational graph. Subgradients for the McCormick relaxations may be evaluated either by a forward AD mode variant~\cite{mitsos2009mccormick} or a reverse AD mode variant~\cite{beckers2012adjoint}.}

\subsection{Implicit function relaxations}\label{sec:LD}
Definition~\ref{Definition:implicit} below formalizes a generic implicit function considered in the remainder of this article.
\begin{Definition}[from \cite{cao2023general}]\label{Definition:implicit}
	\normalfont
	Consider a residual function $\bff{f}:\R^{n_x}\times\R^{n_p}\to\R^{n_x}$, and the following system of equations:
	\begin{equation}\label{eq:residual}
		\bff{f}(\bff{z},\bff{p})=\bff{0}.
	\end{equation}
	Consider a convex compact set $P\subset\R^{n_p}$, and let $Q\subset P$ be the set of $\bff{p}\in P$ for which \eqref{eq:residual} has at least one solution $\bff{z}$. Suppose that the following conditions hold:
	\begin{enumerate}
		\item The set $Q$ is nonempty, so there is a meaningful \emph{implicit function} $\bff{x}:Q\to\R^{n_x}$ that satisfies:
		\begin{equation}\label{eq:implicit}
			\bff{f}(\bff{x}(\bff{p}),\bff{p})=\bff{0}.
		\end{equation}
		\item There is a known interval $X:=[\bff{x}^\L,\bff{x}^\U]\in\I\R^{n_x}$ for which \eqref{eq:implicit} holds and $\bff{x}(\bff{p})\in X$ for every $\bff{p}\in Q$.
	\end{enumerate}
\end{Definition}

Cao and Khan~\cite{cao2023general} developed a new optimization-based approach for constructing useful convex and concave relaxations for the implicit function in Definition~\ref{Definition:implicit}, based on known relaxations of the residual function $\bff{f}$. This approach is summarized in the following proposition.
\begin{proposition}[from \cite{cao2023general}]\label{Proposition:huiyi}
	Consider the setup in Definition~\ref{Definition:implicit}. Let $\bff{f}^\cv,\bff{f}^\cc:X\times P\to\R^{n_x}$ be convex and concave relaxations of $\bff{f}$ on $X\times P$, respectively. Define $\bff{x}^\cv,\bff{x}^\cc:P\to\bar{\R}^{n_x}$ such that, for each $i\in\{1,...,n_x\}$ and $\bff{p}\in P$,
	\begin{equation}\label{opt:convexrelaxation}
		\begin{aligned}
			x_i^\cv(\bff{p})&:=\inf_{\bm{\xi}\in X}\xi_i\quad{\text{s.t.}}\quad \bff{f}^\cv(\bm{\xi},\bff{p})\le \bff{0}\le \bff{f}^\cc(\bm{\xi},\bff{p}),\\
			x_i^\cc(\bff{p})&:=\sup_{\bm{\xi}\in X}\xi_i\quad{\text{s.t.}}\quad \bff{f}^\cv(\bm{\xi},\bff{p})\le \bff{0}\le \bff{f}^\cc(\bm{\xi},\bff{p}).
		\end{aligned}
	\end{equation}
	If these optimization problems are infeasible, then set $x_i^\cv(\bff{p}):=+\infty$ and $x_i^\cc:=-\infty$ for each $i$ by convention. Then, $\bff{x}^\cv$ is a convex relaxation of $\bff{x}$ on $P$, and $\bff{x}^\cc$ is a concave relaxation of $\bff{x}$ on $P$.
\end{proposition}
Observe that for each $\bff{p}\in P$, $x_i^\cv(\bff{p})$ and $x_i^\cc(\bff{p})$ are each described as  optimal-value functions (in the sense of~\cite{thibault1991subdifferentials}) of convex optimization problems. 

\subsection{Lexicographic differentiation}
This subsection introduces Nesterov's lexicographic derivatives (L-derivatives) for nonsmooth functions. These L-derivatives are  subgradients if the underlying function is convex.
\begin{Definition}[from \cite{nesterov2005lexicographic}]\label{Definition:lsmooth}
	\normalfont
	Let $X\subset\R^n$ be open and $\bff{f}:X\to\R^m$ be locally Lipschitz continuous and directionally differentiable at $\bff{x}\in X$. The function $\bff{f}$ is \emph{lexicographically smooth (L-smooth)} at $\bff{x}$ if and only if for any $r\in\mathbb{N}$ and any matrix $\bff{M}:=[\bff{m}_{(1)}\;\bff{m}_{(2)}\;\cdot\cdot\cdot\;\bff{m}_{(r)}]\in\R^{n\times r}$, the following homogenization sequence of functions is well defined:
	\begin{equation}\label{sequence}
		\begin{aligned}
			&\bff{f}^{(0)}_{\bff{x},\bff{M}}:\R^n\to\R^m:\bff{d}\mapsto\bff{f}'(\bff{x};\bff{d}),\\
			&\bff{f}^{(j)}_{\bff{x},\bff{M}}:\R^n\to\R^m:\bff{d}\mapsto[\bff{f}^{(j-1)}_{\bff{x},\bff{M}}]'(\bff{m}_{(j)};\bff{d}),\quad\forall j\in\{1,...,r\}.\\
		\end{aligned}
	\end{equation}
	When the columns of $\bff{M}$ span $\R^n$, $\bff{f}^{(r)}_{\bff{x},\bff{M}}$ is guaranteed to be linear. If $r=n$ and $\bff{M}\in\R^{n\times n}$ is non-singular, then a \emph{lexicographic derivative (L-derivative)} of $\bff{f}$ at $\bff{x}$ in the directions $\bff{M}$, denoted as $\bff{J}_\mrm{L}\bff{f}(\bff{x};\bff{M})$, is defined as
	\begin{equation*}
		\bff{J}_\mrm{L}\bff{f}(\bff{x};\bff{M}):=\bff{J}\bff{f}^{(n)}_{\bff{x},\bff{M}}(\bff{0})\in\R^{m\times n},
	\end{equation*} 
	and the \emph{lexicographic subdifferential} $\partial_\mrm{L}\bff{f}(\bff{x})$ is defined as
	\begin{equation*}
		\begin{aligned}
			\partial_\mrm{L}\bff{f}(\bff{x}):=\left\{\bff{J}_\mrm{L}\bff{f}(\bff{x};\bff{M})\in\R^{m\times n}:\forall\text{ nonsingular }\bff{M}\in\R^{n\times n}\right\}.
		\end{aligned}
	\end{equation*}
\end{Definition}
If $\bff{f}$ is scalar-valued and convex, then we have $\partial_\mrm{L}\bff{f}(\bff{x})\subset\partial\bff{f}(\bff{x})$ (see \cite[Theorem~11]{nesterov2005lexicographic}), and thus any transposed L-derivative is a valid subgradient. The following definition of \emph{lexicographic directional derivative} provides a convenient way for evaluating an L-derivative. 
\begin{Definition}[from \cite{KhanBartonFwdAD}]\label{Definition:LD}
	\normalfont
	Given an open set $X\subset \R^n$, a locally Lipschitz continuous function $\bff{f}:X\to\R^m$ that is L-smooth at $\bff{x}\in X$, and a matrix $\bff{M}:=[\bff{m}_{(1)}\;\cdot\cdot\cdot\;\bff{m}_{(r)}]\in\R^{n\times r}$, define the \emph{lexicographic directional derivative (LD-derivative)} as
	\begin{equation*}
		\begin{aligned}
			\bff{f}'(\bff{x};\bff{M})&:=\left[\bff{f}^{(0)}_{\bff{x},\bff{M}}(\bff{m}_{(1)})\;\bff{f}^{(1)}_{\bff{x},\bff{M}}(\bff{m}_{(2)})\;\cdot\cdot\cdot\;\bff{f}^{(r-1)}_{\bff{x},\bff{M}}(\bff{m}_{(r)})\right].
		\end{aligned}
	\end{equation*}
\end{Definition}
Given $\bff{f}'(\bff{x};\bff{M})$ with square nonsingular $\bff{M}$, an L-derivative $\bff{J}_\mrm{L}\bff{f}(\bff{x};\bff{M})$ can be computed by solving a linear equation system:
\begin{equation}\label{LDL}
	\bff{f}'(\bff{x};\bff{M})=\bff{J}_\mrm{L}\bff{f}(\bff{x};\bff{M})\,\bff{M}.
\end{equation}
When $\bff{M}$ is an identity matrix, this linear equation system is trivial.

We presented a forward AD mode for lexicographic derivative evaluation at the AD2012 conference~\cite{KhanBartonAD,KhanBartonFwdAD}, and a ``branch-locking'' variant at the AD2016 conference~\cite{khan2018branch} that borrows the benefits of the reverse AD mode and handles implicit functions defined in terms of closed-form composite residual functions.
\section{Outline of new subgradient formulations}\label{sec:outline}
The goal of this article is to propose the first approaches for computing subgradients of the implicit function relaxations outlined in Proposition~\ref{Proposition:huiyi}, for aiding in global optimization with implicit function embedded. Our goal is to provide subgradient-based AD rules for implicit function relaxations that are compatible with the subgradient-propagating forward AD mode~\cite{mitsos2009mccormick}.

Our new subgradient evaluation approaches are tailored to different cases, depending on the dimensions of $\bff{x}$ and $\bff{p}$, while making only mild technical assumptions.

Specifically, our contributions are as follows:

\begin{itemize}
	\item \emph{For a scalar-valued implicit function ($n_x=1$):} Here, we demonstrate that when piecewise affine relaxations for the residual function are known, then both the implicit function relaxations and their corresponding subgradients can be efficiently computed in closed form, regardless of the dimension of $\bff{p}$.
	\item \emph{For a vector-valued implicit function ($n_x>1$) with low dimensional parameters ($n_p\le 2$):} In these cases, we propose computing valid subgradients using the directional derivatives of the relaxations. We will introduce a new result showing that the directional derivatives can be computed by solving a linear program (LP).
	\begin{itemize}
		\item For $n_p = 1$, the directional derivative in either direction is a subgradient.
		\item For $n_p = 2$,  subgradients are derived from the directional derivatives in the four cardinal directions.
	\end{itemize}
	\item \emph{For a vector-valued implicit function ($n_x>1$) with high-dimensional parameters ($n_p> 2$):} Here, we calculate a subgradient as an L-derivative (as in Section~\ref{sec:LD}) of the relaxations, by solving a sequence of LPs.
\end{itemize}
\section{Subgradients for ${n_x=1}$}\label{sec:nx=1}
In this section, we consider the case where $n_x = 1$ and \emph{piecewise affine} relaxations $f^\cv$ and $f^\cc$ are employed, as formalized in the following Assumption~\ref{Assumption:special}. Piecewise affine relaxations are commonly used and can be constructed from subgradients~\cite{cao2019convergence} or black-box sampling~\cite{song2021bounding}. It will be shown that in this case, the optimization problems in \eqref{opt:convexrelaxation} are single-variable box-constrained problems, and the implicit function relaxations and the subgradients can be efficiently evaluated in closed form. This approach is inspired from \cite{yuan2023automatic} which studied convex optimal-value functions with box constraints.
\begin{Assumption}\label{Assumption:special}
	\normalfont
	Consider the setup in Definition~\ref{Definition:implicit} and suppose that $n_x=1$. Consider constructing implicit function relaxations ${x}^\cv$ and ${x}^\cc$ following Proposition~\ref{Proposition:huiyi}. Suppose that the relaxations ${{f}}^\cv$ and ${{f}}^\cc$ for ${{f}}$ on $X\times P$ are piecewise affine~\cite{scholtes2012introduction}. That is, for each $(\xi,\bff{p})\in X\times P$, these functions have the form:
	\begin{equation*}
		\begin{aligned}
			f^\cv(\xi,\mathbf{p}) &:= \max\left\{
			\begin{aligned}
				[\mathbf{a}^{\cv,i}]^\top \mathbf{p}
				+ \alpha^{\cv,i}\xi + b^{\cv,i} : i\in\{1,\ldots,k\}
			\end{aligned}
			\right\},\\
			f^\cc(\xi,\mathbf{p}) &:= \min\left\{
			\begin{aligned}
				[\mathbf{a}^{\cc,j}]^\top \mathbf{p}
				+ \alpha^{\cc,j}\xi + b^{\cc,j} :j\in\{1,\ldots,\ell\}
			\end{aligned}
			\right\}.
		\end{aligned}
	\end{equation*}
	
	\ysnote{Let $K^-$ denote the set of indices $i\in\{1,\ldots,k\}$ for which $\alpha^{\cv,i}<0$, and let $K^+$ denote the set of indices $i\in\{1,\ldots,k\}$ for which $\alpha^{\cv,i}>0$. Let $L^-$ denote the set of indices $j\in\{1,\ldots,\ell\}$ for which $\alpha^{\cc,j}<0$, and let $L^+$ denote the set of indices $j\in\{1,\ldots,\ell\}$ for which $\alpha^{\cc,j}>0$.}
\end{Assumption}

Now, under Assumption~\ref{Assumption:special}, define functions \ysnote{$h^\cv,h^\cc:{P}\to\R$} such that for each $\bff{p}\in P$,
\ysnote{
	\begin{equation}\label{def:g}
		\begin{aligned}
			h^\cv(\bff{p})&:=\max\left(
			\{g_i(\bff{p}): i\in K^-\}
			\cup \{\gamma_j(\bff{p}): j\in L^+\}
			\right), \\
			h^\cc(\bff{p})&:=\min\left(
			\{g_i(\bff{p}): i\in K^+\}
			\cup \{\gamma_j(\bff{p}): j\in L^-\}
			\right),
		\end{aligned}
	\end{equation}
	where for each index $i$ and $j$:
	\begin{equation}\label{def:gi}
		\begin{aligned}
			g_i(\bff{p})&:=\frac{-1}{\alpha^{\cv,i}}\left({[\mathbf{a}^{\cv,i}]^\top\mathbf{p}}+ b^{\cv,i}\right), \\
			\gamma_j(\bff{p})&:=\frac{-1}{\alpha^{\cc,j}}\left({[\mathbf{a}^{\cc,j}]^\top\mathbf{p}}+ b^{\cc,j}\right).
		\end{aligned}
\end{equation}}

Observe that \ysnote{$h^\cv$} is a convex piecewise affine function and \ysnote{$h^\cc$} is a concave piecewise affine function.
Then, the optimization problems for computing the implicit function relaxations in ~\eqref{opt:convexrelaxation} can be reformulated to be box-constrained as follows: for each $\bff{p}\in Q$,
\begin{equation}\label{optspecial}
	\begin{aligned}
		x^\cv(\bff{p})&\equiv\min_{{\xi}\in X}\xi\quad\text{s.t.}\quad h^\cv(\bff{p})\le \xi\le h^\cc(\bff{p}),\\
		x^\cc(\bff{p})&\equiv\max_{{\xi}\in X}\xi\quad\text{s.t.}\quad h^\cv(\bff{p})\le \xi\le h^\cc(\bff{p}),
	\end{aligned}
\end{equation}
which has the following closed form:
\begin{equation}\label{optspecial2}
	\begin{aligned}
		x^\cv(\bff{p})&\equiv\max\left\{x^\L,h^\cv(\bff{p})\right\},\\
		x^\cc(\bff{p})&\equiv\min\left\{x^\U,h^\cc(\bff{p})\right\},
	\end{aligned}
\end{equation}
with $X:=[x^\L,x^\U]$. 

The subsequent result follows immediately, and evaluates subgradients of $x^\cv$ and $x^\cc$ based on \eqref{optspecial2}.
\begin{proposition}\label{prop:special}
	Suppose that Assumption~\ref{Assumption:special} holds and consider the formulation~\eqref{optspecial2} where the functions $h^\cv$ and $h^\cc$ are defined in \eqref{def:g} and \eqref{def:gi}. Then, for an arbitrary $\bff{p}\in \mathrm{int}(Q)$, a subgradient $\bff{s}^\cv\in\R^{n_p}$ of $x^\cv$ at $\bff{p}$ is given by:
	\ysnote{
		\begin{itemize}
			\item $\bff{s}^\cv :=\bff{0}$ if $ x^\cv(\bff{p}) = x^\L$; 
			\item otherwise, if $x^\cv(\bff{p})=g_i(\bff{p})$ for some $i\in K^-$, then $\bff{s}^\cv := -\frac{\bff{a}^{\cv,i}}{\alpha^{\cv,i}}$;
			\item otherwise, there exists $j\in L^+$ such that $x^\cv(\bff{p})=\gamma_j(\bff{p})$ and we have $\bff{s}^\cv := -\frac{\bff{a}^{\cc,j}}{\alpha^{\cc,j}}$.  
		\end{itemize}	
		% 		consider any $i\in\{1,...,k^-+l^+\}$ such that $x^\cv(\bff{p})=g^{\cv,i}(\bff{p})$,
		% 		\begin{itemize}
			% 			\item $\bff{s}^\cv := -\frac{\bff{a}^{\cv,i}}{\alpha^{\cv,i}}$, if $i\in\{1,...,k^-\}$,
			% 			\item $\bff{s}^\cv := -\frac{\bff{a}^{\cc,i+v}}{\alpha^{\cc,i+v}}$, if $i\in\{k^-+1,...,k^-+l^+\}$.
			% 		\end{itemize}  
		%% 		$\nabla g^{\cv,i}(\bff{p})$ for some $i\in\{1,...,k^-+l^+\}$ and $x^\cv(\bff{p})=g^{\cv,i}(\bff{p})$.
		% 	\end{itemize}
	Similarly, a subgradient $\bff{s}^\cc\in\R^{n_p}$ of $x^\cc$ at $\bff{p}$ is given by:
	\begin{itemize}
		\item $\bff{s}^\cc:=\bff{0}$ if $ x^\cc(\bff{p}) = x^\U$; 
		\item otherwise, if $x^\cc(\bff{p})=\gamma_j(\bff{p})$ for some $j\in L^-$, then $\bff{s}^\cv := -\frac{\bff{a}^{\cc,j}}{\alpha^{\cc,j}}$;
		\item otherwise,
		there exists $i\in K^+$ such that $x^\cc(\bff{p})=g_i(\bff{p})$ and we have $\bff{s}^\cv := -\frac{\bff{a}^{\cv,i}}{\alpha^{\cv,i}}$.

		% 		\item and otherwise consider any $i\in\{1,...,l^-+k^+\}$ such that $x^\cc(\bff{p})=g^{\cc,i}(\bff{p})$,
		% 		\begin{itemize}
			% 			\item $\bff{s}^\cc:=-\frac{\bff{a}^{\cc,i}}{\alpha^{\cc,i}}$, if $i\in\{1,...,l^-\}$,
			% 			\item $\bff{s}^\cc =  -\frac{\bff{a}^{\cv,i-v}}{\alpha^{\cv,i-v}}$, if $i\in\{l^-+1,...,l^-+k^+\}$.
			% 		\end{itemize}
		% 		$\nabla g^{\cc,i}(\bff{p})$ for some $i\in\{1,...,k^-+l^+\}$ and $x^\cc(\bff{p})=g^{\cc,i}(\bff{p})$.
\end{itemize}}
\end{proposition}
\section{Subgradients for ${n_x>1}$}\label{sec:nx>1}
This section considers the case where $n_x>1$. This case is further divided into three scenarios depending on the dimension of $\bff{p}$, namely $n_p=1$, $n_p=2$, and $n_p>2$. We restrict ourselves to the following setting for the underlying implicit function relaxation system.
\begin{Assumption}\label{Assumption:slater}
	\normalfont
	Consider the setup in Definition~\ref{Definition:implicit}, and consider constructing implicit function relaxations $\bff{x}^\cv$ and $\bff{x}^\cc$ following Proposition~\ref{Proposition:huiyi}. Without loss of generality, assume that $\bff{f}\equiv(\tilde{\bff{f}},\bff{h})$ where the functions $\tilde{\bff{f}}:\R^{n_x}\times\R^{n_p}\to\R^{n_{\tilde{f}}}$ and  $\bff{h}:\R^{n_x}\times\R^{n_p}\to\R^{n_h}$ satisfy that $\bff{h}$ is affine and for each $i\in\{1,...,n_{\tilde{f}}\}$, $\tilde{f}_i$ is not affine. Suppose that for each $i\in\{1,...,n_{\tilde{f}}\}$, the convex and concave relaxations $\tilde{{f}}_i^\cv,\tilde{{f}}_i^\cc$ for $\tilde{{f}}_i$ on $X\times P$ are piecewise differentiable in the sense of Scholtes~\cite{scholtes2012introduction}; that is there exist continuously differentiable convex functions $\tilde{{f}}_i^{\cv,1},\tilde{{f}}_i^{\cv,2},...,\tilde{{f}}_i^{\cv,k_i}:\R^{n_x}\times \R^{n_p}\to\R$ such that for each $(\bm{\xi},\bff{p})\in X\times P$,
	%,\tilde{f}_i^{\cv,2}(\bm{\xi},\bff{p}),...,\tilde{f}_i^{\cv,k_i}(\bm{\xi},\bff{p})
	\begin{equation*}
		\tilde{f}_i^{\cv}(\bm{\xi},\bff{p})\equiv\max\left\{\tilde{f}_i^{\cv,j}(\bm{\xi},\bff{p}): j\in\{1,...,k_i\}\right\},
	\end{equation*}
	and similarly, there exist continuously differentiable concave functions $\tilde{{f}}_i^{\cc,1},\tilde{{f}}_i^{\cc,2},...,\tilde{{f}}_i^{\cv,l_i}:\R^{n_x}\times \R^{n_p}\to\R$ such that
	\begin{equation*}
		\tilde{f}_i^{\cc}(\bm{\xi},\bff{p})\equiv\min\left\{\tilde{f}_i^{\cc,j}(\bm{\xi},\bff{p}):j\in\{1,...,l_i\}\right\}.
	\end{equation*}
	%	\begin{equation*}
		%		\tilde{f}_i^{\cc}(\bm{\xi},\bff{p})\equiv\min(\tilde{f}_i^{\cc,1}(\bm{\xi},\bff{p}),\tilde{f}_i^{\cc,2}(\bm{\xi},\bff{p}),...,\tilde{f}_i^{\cc,l_i}(\bm{\xi},\bff{p})).
		%	\end{equation*}
\end{Assumption}

\ysnote{Under Assumption~\ref{Assumption:slater} and Proposition~\ref{Proposition:huiyi}}, we have for each $i\in\{1,...,n_x\}$ and $\bff{p}\in Q$:
\begin{equation}\label{opt:convexrelaxationreform}
	\begin{aligned}
		x_{{i}}^\cv({\bff{p}})\equiv\min_{\bm{\xi}\in X}\quad&\xi_{{i}}\\
		\text{s.t.}\quad&\bff{h}(\bm{\xi},{\bff{p}})=\bff{0},\\&\forall \tilde{i}\in\{1,...,n_{\tilde{f}}\},\\
		&\tilde{f}_{\tilde{i}}^{\cv,j}(\bm{\xi},{\bff{p}})\le 0,\quad\forall j\in\{1,...,k_{\tilde{i}}\},\\
		&\tilde{f}_{\tilde{i}}^{\cc,j}(\bm{\xi},{\bff{p}})\ge 0,\quad\forall j\in\{1,...,l_{\tilde{i}}\}.		
	\end{aligned}
\end{equation}
Similarly, $x_{{i}}^\cc(\bff{p})$ is constructed using the above formulation with ``max'' in place of ``min''. For notational simplicity, define a function  $\bff{g}:\R^{n_x}\times \R^{n_p}\to\R^{n_g}$ with $n_g:=\sum_{i=1}^{n_{\tilde{f}}}\left(k_i+l_i\right)$ to achieve a compact representation of all inequality constraints in \eqref{opt:convexrelaxationreform} as follows:
\begin{equation}\label{opt:convexrelaxationreform1}
	\begin{aligned}
		x_{{i}}^\cv({\bff{p}})\equiv\min_{\bm{\xi}\in \R^{n_x}}\quad&\xi_{{i}}\\
		\text{s.t.}\quad&\bff{h}(\bm{\xi},{\bff{p}})=\bff{0},\\
		&\bff{g}(\bm{\xi},\bff{p})\le \bff{0}.
	\end{aligned}
\end{equation}
Note that the  ``$\bm{\xi}\in X$'' constraint is incorporated into the vector function $\bff{g}$.
Denote the $\bff{p}$-dependent feasible region of the optimization problem above as $\Omega(\bff{p})$. We will enforce the following assumption as well.
\ysnote{
	\begin{Assumption}\label{Assumption:slater2}
		\normalfont
		Suppose that Assumption~\ref{Assumption:slater} holds, and consider some $\hat{\bff{p}}\in\mathrm{int}(Q)$. Since $\bff{h}$ is affine, let $\bff{a}_i$ denote the constant gradient of each component $h_i$ with respect to $\bm{\xi}\in\R^{n_x}$. Suppose that the vectors $\bff{a}_i$ for $i\in\{1,...,n_h\}$ are linearly independent. Suppose that there exists a neighborhood $N_{\hat{\bff{p}}}\subset Q$ of $\hat{\bff{p}}$ such that for each $\bff{p}\in N_{\hat{\bff{p}}}$, there exists $\bm{\xi}_\bff{p}\in X$ (dependent on $\bff{p}$) such that ${\bff{h}}(\bm{\xi},{\bff{p}})=\bff{0}$ and $\bff{g}(\bm{\xi},\bff{p})<\bff{0}$. 
\end{Assumption}}

Assumptions~\ref{Assumption:slater}~and~\ref{Assumption:slater2} impose two additional requirements on the implicit function relaxation system. The first is that the employed relaxations for $\bff{f}$ must be piecewise differentiable. Besides the piecewise affine relaxations mentioned in Assumption~\ref{Assumption:special}, many established convex relaxations are piecewise differentiable, such as McCormick relaxations~\cite{mccormick1976computability,scott2011generalized}, $\alpha$BB relaxations~\cite{adjiman1998global1}, and convex envelopes~\cite{khajavirad2012convex,khajavirad2013convex}. The second requirement is the strong Slater condition, which is commonly enforced when guaranteeing strong duality in convex optimization~(see \cite{boyd2004convex}). %Suppose that Assumption~\ref{Assumption:slater} holds throughout this section.

Throughout this section, for ease of exposition, consider an arbitrary $i\in\{1,...,n_x\}$ and consider computing a subgradient of $x_i^\cv$ at $\hat{\bff{p}}$ under Assumption~\ref{Assumption:slater2}. The new subgradient evaluation approaches can be readily applied to the concave implicit function relaxation, with $-x^\cc_i$ in place of $x_i^\cv$.

The following result shows that a directional derivative of $x_i^\cv$ at $\hat{\bff{p}}$ can be computed by constructing and solving an auxiliary linear program (LP). This result serves as the foundation of all subsequent subgradient evaluation approaches in this section. 
\begin{proposition}\label{prop:direct}
	Under Assumption~\ref{Assumption:slater2}, the function $x_i^\cv$ is locally Lipschitz continuous and directionally differentiable at $\hat{\bff{p}}$. Moreover, consider an arbitrary optimal solution $\hat{\bm{\xi}}$ of the optimization problem in \eqref{opt:convexrelaxationreform1}. Then, for each direction $\bff{d}\in\R^{n_p}$,
	\begin{equation}\label{opt:direct}
		\begin{aligned}
			[x_i^\cv]'(\hat{\bff{p}};\bff{d})&\equiv\min_{\bff{w}\in\R^{n_x}}\quad w_i\\
			\text{s.t.}\quad &[\nabla_\bff{x}g_i(\hat{\bm{\xi}},\hat{\bff{p}})]^\T\bff{w}\le-[\nabla_\bff{y}g_i(\hat{\bm{\xi}},\hat{\bff{p}})]^\T\bff{d},\quad\forall i\in\{1,...,n_g\}\text{ such that }g_i(\hat{\bm{\xi}},\hat{\bff{p}})=0,\\
		%	&\quad\quad,\\
			&[\nabla_\bff{x}h_j(\hat{\bm{\xi}},\hat{\bff{p}})]^\T\bff{w}=-[\nabla_\bff{y}h_j(\hat{\bm{\xi}},\hat{\bff{p}})]^\T\bff{d},\quad\forall j\in\{1,...,l\}.\\
		%	&\quad\quad\quad\quad\quad\quad\quad\quad\quad\quad\quad\quad\,
		\end{aligned}
	\end{equation}
\end{proposition}
\begin{proof}
	This result follows directly from \cite[Theorem~3~and~Proposition~1]{song2024generalized}.
\end{proof}
\begin{Remark}
	\normalfont
	We highlight that \eqref{opt:direct} admits any optimal solution $\hat{\bm{\xi}}$. This is particularly beneficial for implementation, because it means that any optimal solution returned by a numerical optimization solver can be used for computing the directional derivative.
\end{Remark}
% and we aim at computing subgradients for an arbitrary

Now, we consider cases separately depending on the value of $n_p$.
\subsection{The case in which ${n_p=1}$:}
In this case, the directional derivative along the direction $1$ or $-1$ itself is a valid subgradient, which is formalized below.

\begin{proposition}
	Suppose that Assumption~\ref{Assumption:slater2} holds and $n_p=1$, then for an arbitrary $d\in\{1,-1\}$, the directional derivative $x_i^\cv(\hat{p};d)$ described in Proposition~\ref{prop:direct} is a valid subgradient of $x_i^\cv$ at $\hat{p}$. 
\end{proposition}

\subsection{The case in which ${n_p=2}$:}\label{sec:np2}
The following result is obtained by specializing \cite[Theorem~3.3]{yuan2020constructing} to the implicit function relaxation $x_i^\cv$, and shows that a valid subgradient can be computed from directional derivatives from the four cardinal directions. Recall a directional derivative can be evaluated by solving an auxiliary LP outlined in Proposition~\ref{prop:direct}.
\begin{proposition}\label{prop:np2}
	Suppose that Assumption~\ref{Assumption:slater2} holds and $n_p=2$. Then, a subgradient of $x_i^\cv$ at $\hat{\bff{p}}$ is given by:
	\begin{equation*}
		\frac{1}{2}	\begin{bmatrix}
			[x_i^\cv]'(\hat{\bff{p}};\bff{e}^{(1)})-[x_i^\cv]'(\hat{\bff{p}};-\bff{e}^{(1)})\\
			[x_i^\cv]'(\hat{\bff{p}};\bff{e}^{(2)})-[x_i^\cv]'(\hat{\bff{p}};-\bff{e}^{(2)}).
		\end{bmatrix}
	\end{equation*} 
\end{proposition}
\begin{proof}
	This result follows directly from \cite[Corollary~3.6]{yuan2020constructing}.
\end{proof}
\subsection{The case in which ${n_p>2}$:}\label{sec:LDapproach}
This case is more complicated than the previous two, as the subgradient evaluation methods proposed previously are not applicable. Therefore, we shift our attention to computing an L-derivative of $x_i^\cv$ at $\hat{\bff{p}}$. Since $x_i^\cv$ is convex, an L-derivative is a valid subgradient. Roughly, this is done by specializing relevant results in \cite{song2024generalized}, which provided nonsmooth forward AD rules of parameterized convex programs.

%Observe that \eqref{opt:direct} contains a parameterized linear program (LP) where the parameter $\bff{d}$ is at the constraints' right-hand side.
The key to compute an L-derivative is to construct the homogenization function sequence~\eqref{sequence}, starting from the directional derivative. Note that the directional derivative of $x_i^\cv$ has already been established in Proposition~\ref{prop:direct}. To proceed, we simplify the notation in \eqref{opt:direct} as follows:
\begin{equation}\label{eq:psi}
	\begin{aligned}
		[x_i^\cv]'(\hat{\bff{p}};\bff{d})\equiv \min_{\bff{w}\in\R^{n_x}}\quad &[\bff{e}^{(i)}]^\T \bff{w}\\
		\text{s.t.}\quad &\bff{A}\bff{w}\le\bff{G}^\A\bff{d},\\
		&\bff{B}\bff{w}=\bff{G}^\B\bff{d},
	\end{aligned}
\end{equation}
where $\bff{A}\in\R^{n_A\times n_x},\bff{B}\in\R^{n_B\times n_x},\bff{G}^\A\in\R^{n_A\times n_p},\bff{G}^\B\in\R^{n_B\times n_p}$ are straightforwardly constructed based on the constraints in \eqref{opt:direct}. Based on \eqref{eq:psi}, the following result evaluates an LD-derivative of $x_i^\cv$ by solving a sequence of LPs. 
\begin{theorem}\label{Theorem:marginal}
	Suppose that Assumption~\ref{Assumption:slater2} holds and consider all quantities defined in \eqref{eq:psi}. Then, the function $x_i^\cv$ is L-smooth at $\hat{\bff{p}}$, and for any $\bff{M}\in\R^{n_p\times r}$, the following sequence of functions 
	\begin{equation*}
		\begin{aligned}
			&[x_i^\cv]^{(0)}_{\hat{\bff{p}},\bff{M}}:{\R^{n_p}\to\R}:\bff{d}\mapsto[x_i^\cv]'(\hat{\bff{p}};\bff{d}),\\
			&[x_i^\cv]^{(j)}_{\hat{\bff{p}},\bff{M}}:{\R^{n_p}\to\R}:\bff{d}\mapsto\left[[x_i^\cv]^{(j-1)}_{\hat{\bff{p}},\bff{M}}\right]'(\bff{m}_{(j)};\bff{d}),\quad 	\forall j\in\{1,...,r\},\\
		\end{aligned}
	\end{equation*}
	are well-defined and given by: for each $j\in\{0,...,r\}$,
	\begin{equation}\label{sequencepsi}
		\begin{aligned}
			[x_i^\cv]^{(j)}_{\hat{\bff{p}},\bff{M}}:&\R^{n_p}\to\R:{\bff{d}}\mapsto\max_{\bm{\lambda}^\A,\bm{\lambda}^\B}\quad[\bff{G}{\bff{d}}]^\T\bm{\lambda}\\ %+ [\bff{G}^\B\tilde{\bff{d}}]^\T\bm{\lambda}^\B\\
			\text{s.t.}\quad& \bff{A}^\T\bm{\lambda}^\A + \bff{B}^\T\bm{\lambda}^\B = \bff{e}^{(i)},\\
			&\bm{\lambda}^\A\le \bff{0},\\
			&\begin{bmatrix}
				-[\bff{G}\bff{m}_{(1)}]^\T\\
				-[\bff{G}\bff{m}_{(2)}]^\T\\
				\vdots\\
				-[\bff{G}\bff{m}_{(j)}]^\T
			\end{bmatrix}\bm{\lambda}\le 
			\begin{bmatrix}
				-[x_i^\cv]_{\hat{\bff{p}},\bff{M}}^{(0)}(\bff{m}_{(1)})\\
				-[x_i^\cv]_{\hat{\bff{p}},\bff{M}}^{(1)}(\bff{m}_{(2)})\\
				\vdots\\
				-[x_i^\cv]_{\hat{\bff{p}},\bff{M}}^{(j-1)}(\bff{m}_{(j)})
			\end{bmatrix},
		\end{aligned}
	\end{equation}
	with $\bm{\lambda} := \begin{bmatrix}
		\bm{\lambda}^\A\\
		\bm{\lambda}^\B
	\end{bmatrix}$ and $\bff{G}:=\begin{bmatrix}
		\bff{G}^\A\\\bff{G}^\B
	\end{bmatrix}$. For $j=0$, the constraint $-[\bff{G}\bff{m}_{(j)}]^\T\bm{\lambda}\le-[x_i^\cv]_{\hat{\bff{p}},\bff{M}}^{(j-1)}(\bff{m}_{(j)})$ is understood to be vanished. Moreover, associated to the sequence~\eqref{sequencepsi}, define a sequence of sets $D_{\hat{\bff{p}},\bff{M}}^{(j)}$ such that for each $j\in\{0,...,r-1\}$, $D^{(j)}_{\hat{\bff{p}},\bff{M}}$ is the optimal solution set of the LP for evaluating $[x_i^\cv]^{(j)}_{\hat{\bff{p}},\bff{M}}(\bff{m}_{(j+1)})$ in \eqref{sequencepsi}. Then, it follows that 
	for each $j\in\{1,...,r-1\}$,
	\begin{equation}\label{Drelation}
		D^{(j)}_{\hat{\bff{p}},\bff{M}}\equiv\Big(\arg\max _{\bm{\lambda}\in D^{(j-1)}_{\hat{\bff{p}},\bff{M}}}[\bff{G}\bff{m}_{(j+1)}]^\T\bm{\lambda}\Big)\subset D^{(j-1)}_{\hat{\bff{p}},\bff{M}}.
	\end{equation} 
	Furthermore, the LD-derivative of $x_i^\cv$ at $\hat{\bff{p}}$ in the directions $\bff{M}$ is given by
	\begin{equation}\label{LDexpression}
		[x_i^\cv]'(\hat{\bff{p}};\bff{M})\equiv\bm{\lambda}^\T\bff{G}\bff{M},\quad\text{for any }\bm{\lambda}\in D^{(r-1)}_{\hat{\bff{p}},\bff{M}}.
	\end{equation} 
	%If $\bff{M}\in\R^{n_p\times n_p}$ is an identity matrix, then $\left[[x_i^\cv]'(\hat{\bff{p}};\bff{M})\right]^\T$ is a subgradient of $x_i^\cv$ at $\hat{\bff{p}}$.
\end{theorem}
\begin{proof}
	Observe that for any $\bff{d}\in\R^{n_p}$, the optimization problem for defining $[x_i^\cv]^{(0)}_{\hat{\bff{p}},\bff{M}}$ in \eqref{sequencepsi} is the dual problem of the right-hand-side (RHS) optimization problem in \eqref{eq:psi}, and due to the strong duality of LPs~\cite[Theorem~4.4]{bertsimas1997introduction}, we have $[x_i^\cv]^{(0)}_{\hat{\bff{p}},\bff{M}}(\bff{d})\equiv [x_i^\cv]'(\hat{\bff{p}};\bff{d})$.
	
	Now, consider $j=1$, and for any $\bff{d}\in\R^{n_p}$, according to \eqref{sequencepsi},
	\begin{equation*}
		\begin{aligned}
			[x_i^\cv]^{(1)}_{\hat{\bff{p}},\bff{M}}({\bff{d}}):=\max_{\bm{\lambda}^\A,\bm{\lambda}^\B}\quad&[\bff{G}{\bff{d}}]^\T\bm{\lambda}\\ %+ [\bff{G}^\B\tilde{\bff{d}}]^\T\bm{\lambda}^\B\\
			\text{s.t.}\quad& \bff{A}^\T\bm{\lambda}^\A + \bff{B}^\T\bm{\lambda}^\B = \bff{e}^{(i)},\\
			&\bm{\lambda}^\A\le \bff{0},\\
			&
			-[\bff{G}\bff{m}_{(1)}]^\T
			\bm{\lambda}\le 
			-[x_i^\cv]^{(0)}_{\hat{\bff{p}},\bff{M}}(\bff{m}_{(1)}).
		\end{aligned}
	\end{equation*}
	Observe that the feasible region $D_{\hat{\bff{p}},\bff{M}}^{(1)}$ of the optimization problem above is the optimal dual solution set of the RHS optimization problem in \eqref{eq:psi} with $\bff{d}:=\bff{m}_{(1)}$. Then, since $[x_i^\cv]'(\hat{\bff{p}};\cdot)$ is well-defined on $\R^{n_p}$ and according to \cite[Theorem~3.3.2]{hiriart2013convex}, we have that $[x_i^\cv]'(\hat{\bff{p}};\cdot)$ is a convex function, and $D_{\hat{\bff{p}},\bff{M}}^{(1)}$ is the subdifferential of $[x_i^\cv]'(\hat{\bff{p}};\cdot)$ at $\bff{m}_{(1)}$ and is compact. Then, the expression of $[x_i^\cv]^{(1)}_{\hat{\bff{p}},\bff{M}}({\bff{d}})$ above is exactly the formulation of directional derivatives of a convex function as in \cite[Theorem~23.4]{rockafellar2015convex}. Since $D_{\hat{\bff{p}},\bff{M}}^{(1)}$ is compact, the rest of the sequence $	[x_i^\cv]^{(j)}_{\hat{\bff{p}},\bff{M}},\forall j\in\{2,...,r\}$ follows from \cite[Theorem~4]{song2024generalized}.
	
	Since moreover Proposition~\ref{prop:direct} establishes the local Lipschitz continuity and directional differentiability of $x_i^\cv$, it follows that $x_i^\cv$ is L-smooth at $\hat{\bff{p}}$.
	
	\eqref{Drelation} is indeed valid from observing~\eqref{sequencepsi}. Finally, \eqref{LDexpression} follows from \eqref{Drelation} and the fact that 
	\begin{equation*}
		\begin{aligned}
			[x_i^\cv]'(\hat{\bff{p}};\bff{M})\equiv\left[[x_i^\cv]^{(0)}_{\hat{\bff{p}};\bff{M}}(\bff{m}_{(1)})\;[x_i^\cv]^{(1)}_{\hat{\bff{p}};\bff{M}}(\bff{m}_{(2)})\;\cdot\cdot\cdot\;[x_i^\cv]^{(r-1)}_{\hat{\bff{p}};\bff{M}}(\bff{m}_{(r)})\right].
		\end{aligned}	
	\end{equation*} 
\end{proof}

Theorem~\ref{Theorem:marginal} gives a practical way for computing a subgradient of $x_i^\cv$ at $\hat{\bff{p}}$. Specifically, given $\bff{M}$ whose columns span $\R^{n_p}$, Theorem~\ref{Theorem:marginal} computes the LD-derivative $[x_i^\cv]'(\hat{\bff{p}};\bff{M})$. Then, since the columns of $\bff{M}$ span $\R^{n_p}$, an L-derivative $\bff{J}_\L x_i^\cv(\hat{\bff{p}};\bff{M})$, which reduces to a subgradient in this case, can be evaluated by solving the linear equation system~\eqref{LDL}. In particular, when $\bff{M}\in\R^{n_p\times n_p}$ is an identity matrix, we have $\bff{J}_\L x_i^\cv(\hat{\bff{p}};\bff{M})\equiv [x_i^\cv]'(\hat{\bff{p}};\bff{M})$. %Subgradients of $x_i^\cc$ at $\hat{\bff{p}}$ can be obtained by applying Theorem~\ref{Theorem:marginal} to $[-x_i^\cc](\hat{\bff{p}})$. 
\begin{Remark}\label{remark:uniqueness}
	\normalfont
	Theorem~\ref{Theorem:marginal} shows that a subgradient of $x_i^\cv$ at $\hat{\bff{p}}$ may be computed by solving $n_p$ LPs. However, in practice, we probably do not need to solve all $n_p$ LPs. Since \eqref{Drelation} holds, if at some $\hat{j}$, the optimal solution set for the LP defining $[x_i^\cv]^{(\hat{j})}_{\hat{\bff{p}},\bff{M}}(\bff{m}_{(\hat{j}+1)})$ is a singleton, then we do not need to formulate and solve LPs for computing $[x_i^\cv]^{({j})}_{\hat{\bff{p}},\bff{M}}(\bff{m}_{({j}+1)}),\forall j > \hat{j}$. \cite[Proposition~3]{song2024generalized} (summarized from \cite{MANGASARIAN1979151}) examines whether an LP has a unique optimal solution or not by solving an auxiliary LP. Employing this (optimal) solution uniqueness examination method, in most cases, our new subgradient evaluation approach would only solve two LPs, namely the LP for computing $[x_i^\cv]^{(0)}_{\hat{\bff{p}},\bff{M}}(\bff{m}_{(1)})$ and the associated uniqueness examination LP. See \cite[Section~3.1]{song2024generalized} for a detailed explanation.
\end{Remark}
Finally, we note that this L-derivative-based subgradient evaluation approach is applicable regardless of the dimensions of $\bff{x}$ and $\bff{p}$. It may even be favorable over the method in Section~\ref{sec:np2} for $n_p=2$, since it only requires solving two LPs, while the previous method requires solving four. However, implementing the L-derivative-based approach may present additional complexities.

\section{Numerical examples}\label{sec:numerical}
This section presents numerical examples to illustrate our new proposed subgradient AD methods for implicit function relaxations, based on a proof-of-concept implementation in Julia~\cite{bezanson2017julia} that employs the following packages. The package \textsf{McCormick.jl}~v0.13.6 provided within EAGO \cite{wilhelm2022eago} was used to construct convex and concave relaxations and the corresponding subgradients, following either the standard McCormick relaxations~\cite{mccormick1976computability,mitsos2009mccormick,scott2011generalized} or the differentiable McCormick relaxations~\cite{khan2017differentiable}. All optimization problems were formulated via \textsf{JuMP.jl}~v1.13.0~\cite{dunning2017jump}. Convex nonlinear programs were solved using \textsf{IPOPT}~v3.14.13~\cite{wachter2006implementation}, and LPs were solved using \textsf{GLPK.jl} v1.1.3.~\cite{GLPK.jl}. The nonlinear equation system solver \textsf{NLsolve.jl}~v4.5.1~\cite{NLsolve.jl} was employed to evaluate the original implicit function at each parameter value. When necessary, the symbolic expressions of sensitivities of smooth functions were provided by \textsf{Symbolics.jl}~\cite{symbolics}. All solvers were used with their default settings. All computations were performed on a typical laptop.  
\begin{EExample}\label{exp:1}
	\normalfont
	This example considers an implicit function defined by the van der Waals equation of state in thermodynamics, adapted from \cite[Example~2]{cao2023general}. In this example, we will use uppercase letters to denote scalars, for consistency with engineering notational conventions.
	The van der Waals equation of state describes the relationship between pressure $P$ (atm), volume $V$ (L), temperature $T$ (K), and amount $n$ (mol) of non-ideal gases as follows:
	\begin{equation*}
		f(V,P,T):=\left(P+a\frac{n^2}{V^2}\right)(V-nb)-nRT = 0.
	\end{equation*}
	Here $R=0.0820574\;\frac{\text{L atm}}{\text{K mol}}$ is the gas constant, and $a$ and $b$ are van der Waals constants that depend on the gas's chemical composition. Consider $n=1$ mol of carbon dioxide gas with $a = 3.610\;\frac{\text{L}^2\text{ atm}}{\text{mol}^2}$ and $b = 0.0429\;\frac{\text{L}}{\text{mol}}$. Let $P\in [0.5,1.1]$ and $T \in[250, 320]$. Viewing $V$ as an implicit function of $(P,T)$, it is verified that an interval $X:=[10,70]$ encloses all possible values of $V$. With this setup, we aim at constructing convex and concave relaxations of $V$ in terms of $P$ and $T$ following the closed-form reformulation method proposed in Section~\ref{sec:nx=1}. With these closed-form relaxations, the subgradients are readily available as in Proposition~\ref{prop:special}.
	
	We construct the piecewise affine relaxations required in Assumption~\ref{Assumption:special} using \emph{subtangents} of McCormick relaxations of $f$. \ysnote{Analogous to a tangent, a subtangent is a hyperplane that passes through a reference point on the underlying relaxation and has a normal vector equal to some subgradient of the relaxation at that point.} We randomly pick two points $(17.67, 0.68, 274.27)$ and $(67.78, 0.73,288.82)$ on the domain $[10,70]\times[0.5,1.1]\times[250,320]$. Then, a two-piece affine convex relaxation is constructed as the pointwise maximum of the subtangents of the convex McCormick relaxation at these two points:
	\begin{equation*}
		f^\cv(V,P,T) = \max\left\{\begin{aligned}&0.50V + 9.96P-0.08T-4.86,\\
			&1.14V + 69.96P - 0.08T - 79.41
		\end{aligned}\right\},
	\end{equation*} 
	and similarly, a two-piece affine concave relaxation is constructed as the pointwise minimum of two subtangents of the concave McCormick relaxations:
	\begin{equation*}
		f^\cc(V,P,T) = \min\left\{\begin{aligned}&1.13V + 9.95P -0.08T-10.97,\\&0.43V + 69.95P - 0.08T-30.11\end{aligned}\right\}.
	\end{equation*}
	Then, with the functions $g^\cv$ and $g^\cc$ constructed in \eqref{def:g} and \eqref{def:gi}, the implicit function relaxations $V^\cv$ and $V^\cc$ are constructed following \eqref{optspecial2}:
	\begin{equation*}
		\begin{aligned}
			V^\cv(P,T)&:=\max\left\{\begin{aligned}
				&10,\\
				&-8.81P+0.07T+9.71,\\
				&-162.67P+0.19T+70.02
			\end{aligned}\right\},\\
			V^\cc(P,T)&:=\min\left\{\begin{aligned}
				&70,\\
				&-19.92P+0.16T+9.72,\\
				&-61.37P+0.07T+69.66
			\end{aligned}\right\}.
		\end{aligned}
	\end{equation*}
	%Next, following \eqref{optspecial2}, the closed-form convex and concave relaxations of $V$ in terms of $P$ and $T$ are constructed as:
	%\begin{equation*}
	%	\begin{aligned}
		%	V^\cv(P,T)&:=\max\left\{10, g^\cv(P,T)\right\},\\
		%	V^\cc(P,T)&:=\min\left\{70, g^\cc(P,T)\right\}.
		%	\end{aligned}
	%\end{equation*}
	Figure~\ref{fig:exp1} depicts the original implicit function and the constructed convex and concave relaxations. It shows that our method does indeed provide correct piecewise affine relaxations without solving any optimization problems, as guaranteed by the analysis in Section~\ref{sec:nx=1}. Note that with the explicit forms of these relaxations, the corresponding subgradients are readily available through Proposition~\ref{prop:special}.
	\begin{figure}
		
		\centering
		\includegraphics[width=0.4\textwidth]{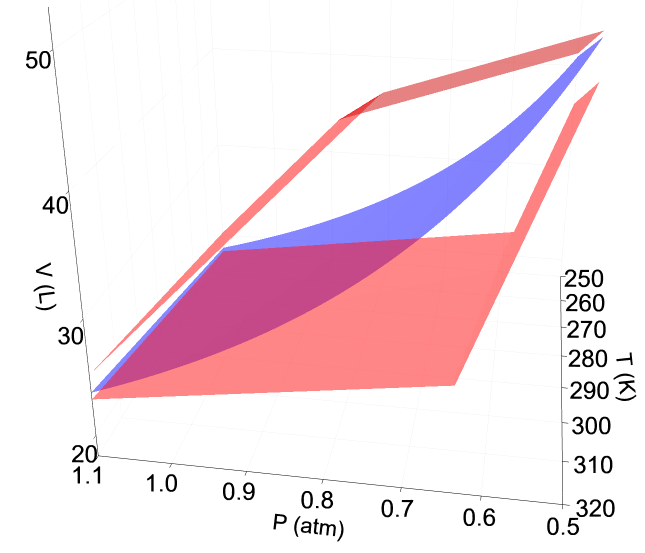}
		\caption{The implicit function $V$ (blue) in terms of $P$ and $T$ on $[0.5,1.1]\times[250,320]$, along with the constructed convex and concave relaxations (red), for Example~\ref{exp:1}.}
		\label{fig:exp1}
	\end{figure}
\end{EExample}
\begin{EExample}\label{exp:2}
	\normalfont
	Consider the following system of equations adapted from \cite[Example~5.2]{stuber2015convex} with three variables $(z_1,z_2,z_3)$ and two uncertain parameters $(p_1, p_2)$:
	%	\begin{equation*}
		%		\begin{aligned}
			%			f_1(\bff{z},\bff{p}):=\;&c(\exp(38z_1)-1)+p_1z_1-1.67z_2\\&+0.69z_3-8.03=0,\\
			%				f_2(\bff{z},\bff{p}):=\;&1.98c(\exp(38z_2)-1)+0.66z_1+p_2z_2\\&+0.66z_3+4.05=0,\\
			%					f_3(\bff{z},\bff{p}):=\;&c(\exp(38z_3)-1)+z_1-z_2\\&+3.7z_3-6.0=0,\\
			%		\end{aligned}
		%	\end{equation*}
	\begin{equation*}
		\begin{aligned}
			c(e^{38z_1}-1)+p_1z_1-1.67z_2+0.69z_3-8.03=0,\\
			1.98c(e^{38z_2}-1)+0.66z_1+p_2z_2+0.66z_3+4.05=0,\\
			c(e^{38z_3}-1)+z_1-z_2+3.7z_3-6.0=0,\\
		\end{aligned}
	\end{equation*}
	where $c = 1.0\times10^{-9}$, $ p_1\in[0.50,0.74]$, and $p_2\in[1.21,1.48]$. An interval that encloses all possible values of $(z_1,z_2,z_3)$ is given by $[0.25,0.70]\times[-5.0,-2.0]\times[0.42,0.53]$. We employ the differentiable McCormick relaxations~\cite{khan2017differentiable} for each residual function, and compute the implicit function relaxations by constructing and solving the optimization problems in \eqref{opt:convexrelaxation}. Then, at any reference point, subgradients of the implicit function relaxations are computed following Proposition~\ref{prop:np2}, where the required directional derivatives are computed by constructing and solving the auxiliary LP~\eqref{opt:direct}. Constructing the auxiliary LP requires identifying all active constraints at an optimal solution of the optimization problem in \eqref{opt:convexrelaxation}. Numerically, we consider a constraint $g_i(\hat{\bm{\xi}},\hat{\bff{p}})=0$ to be active if the value of the residual function $g_i(\hat{\bm{\xi}},\hat{\bff{p}})$ is extremely close to zero, e.g., within a predefined tolerance $1\times10^{-7}$.

	Figure~\ref{fig:exp2} presents a cross-section plot of the original implicit function $x_3$, convex and concave relaxations, and corresponding subtangent lines generated by our new method. Observe that these subtangents do indeed appear to be correct, as is guaranteed by Proposition~\ref{prop:np2}. 
	
	\begin{figure}
		
		\centering
		\includegraphics[width=0.32\textwidth]{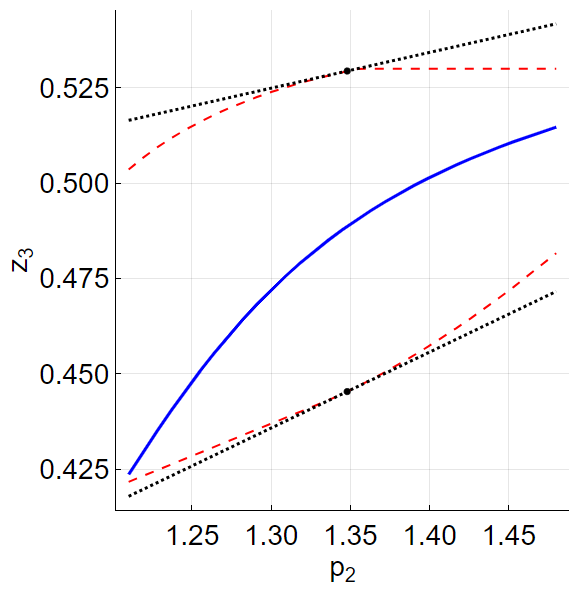}
		\caption{A cross-section at $p_1:=0.6$ of the implicit function $z_3$ (solid blue) of $(p_1,p_2)$, along with convex and concave relaxations (dashed red) and subtangents (dotted black) generated from the reference point $(p_1,p_2):=(0.6,1.348)$, for Example~\ref{exp:2}.}
		\label{fig:exp2}
	\end{figure}
\end{EExample}
\begin{EExample}\label{exp:3}
	\normalfont
	\ysnote{This example applies the L-derivative-based subgradient evaluation method presented in Section~\ref{sec:LDapproach} to a problem in chemical reactor design. Stuber and Barton~\cite{stuber2015semi} studied a nonlinear equation-based model of a continuous-stirred tank reactor for benzene chlorination. They considered the reactor's outputs, including the composition of the product stream and the product stream flow rate, as implicit functions of process uncertainties such as the input flow rate and chemical reaction constants. They deployed this model in a semi-infinite global optimization formulation, whose solution requires generating convex and concave relaxations of the implicit functions. The following equation system is adapted from \cite[Example~3]{stuber2015semi}:
		\begin{equation*}
			\begin{aligned}
				p_3 - z_1z_4 - 15r_1 &= 0,\\
				-z_2z_4 + 15(r_1-r_2) &= 0,\\
				-z_3z_4 + 15r_2 &= 0,\\
				1 - z_1 - z_2 - z_3 &= 0,
			\end{aligned}
		\end{equation*}
		where
		\begin{equation*}
			\begin{aligned}
				r_1 &= p_1z_1/(0.09z_1  + 0.10z_2 + 0.11z_3),\\
				r_2 &= p_2z_2/(0.09z_1  + 0.10z_2 + 0.11z_3).
			\end{aligned}
		\end{equation*}
		In this formulation, the uncertain parameters are {$p_1\in[0.38,0.42]$}, {$p_2\in[0.053,0.058]$}, and {$p_3\in[8,10]$}, and the outputs $z_i$ are known to  have the following bounds: {$z_1\in[0.10,0.16]$}, {$z_2\in[0.40,0.49]$}, {$z_3\in[0.35, 0.45]$}, and {$z_4\in[8,10]$}. Since this example has three uncertain parameters and four outputs, the subgradient evaluation methods employed in the previous two examples are not applicable.

		Following Adjiman et al.~\cite{adjiman1998global1}, we employ the $\alpha$BB relaxations (in their uniform diagonal shift formulation) for each residual function, with the  diagonal shift parameter $\alpha$ computed using the Scaled Gerschgorin Method. Then, convex and concave relaxations for the implicit functions are generated following Proposition~\ref{Proposition:huiyi}. To evaluate subgradients, we construct and solve the sequence of LPs presented in Theorem~\ref{Theorem:marginal} with $\bff{M}:=\bff{I}$. After solving each LP, we solve an auxiliary LP~\cite[Proposition~3]{song2024generalized} to examine the uniqueness of the obtained optimal solution. As discussed in Remark~\ref{remark:uniqueness}, when the solution is unique, there is no need to solve further LPs in the sequence.
		
		Figure~\ref{fig:exp3} illustrates a cross-sectional plot of the implicit function $z_1$, the implicit function relaxations, and two subtangent planes generated from some reference points. Observe that these planes appear to be correct subtangents, as guaranteed by Theorem~\ref{Theorem:marginal}. Moreover, for each reference point, only two LPs were solved to evaluate a subgradient, namely the first LP in the LP sequence and the corresponding uniqueness examination~LP. }
	\begin{figure}
		
		\centering
		\includegraphics[width=.4\textwidth]{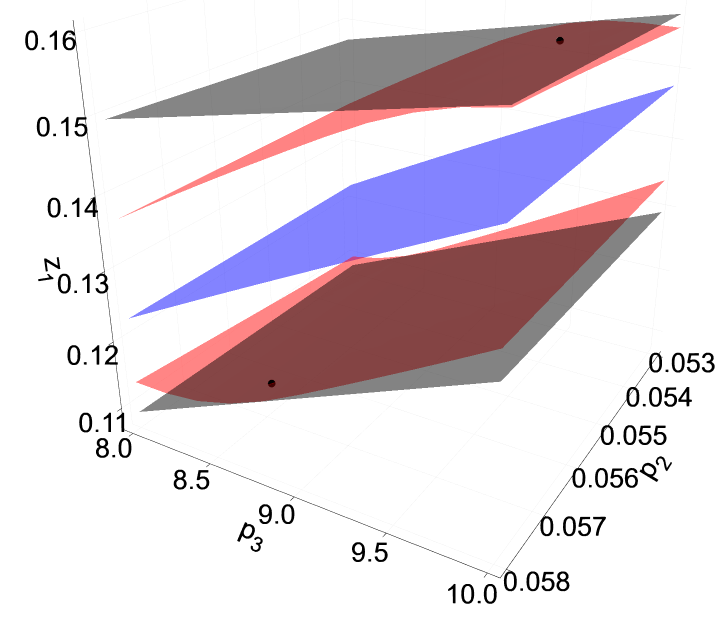}
		\caption{A cross-section at $p_1:=0.40$ of the implicit function $z_1$ (blue) of $(p_1,p_2, p_3)$, along with convex and concave relaxations (red) and  subtangents constructed using our new approach (black), for  Example~\ref{exp:3}. One subtangent is generated using the subgradient of the convex relaxation at $(p_1,p_2,p_3):= (0.40,0.0575,8.7 )$, and another is generated using the subgradient of the concave relaxation at $(p_1,p_2,p_3):= (0.40,0.0545,9.6 )$.}
		\label{fig:exp3}
	\end{figure}
	
\end{EExample}
\section{Conclusions}

This article has presented the first subgradient propagation rules for the implicit function relaxations~\eqref{opt:convexrelaxation}. These rules were specifically developed to be compatible with the forward AD mode for subgradient propagation~\cite{mitsos2009mccormick}, and effectively allow us to include implicit functions in this method's elemental function library, whenever we have sufficient access to well-behaved relaxations of the residual function. \ysnote{Future work may involve applying the new subgradient evaluation approaches within global optimization case studies.}

\end{document}